\documentclass[11pt]{article}
\usepackage{amsmath, amssymb, amsthm}
\usepackage{graphicx}
\usepackage{hyperref}
\usepackage{url}
\usepackage{booktabs}
\usepackage[margin=1in]{geometry}
\usepackage{authblk}
\usepackage{xcolor}
\usepackage{enumitem}

\newtheorem{theorem}{Theorem}[section]

\newtheorem{corollary}[theorem]{Corollary}
\newtheorem{definition}[theorem]{Definition}
\newtheorem{remark}[theorem]{Remark}

\newtheorem{proposition}[theorem]{Proposition}

\title{An Infinite Family of 6-Regular Bi-Cayley Graphs from the Petersen Graph}
\author{Stuart Anderson}
\affil{University of Sydney, Alumni}
\date{\today}

\begin{document}

\maketitle

\begin{abstract}
We construct an infinite family of 6-regular graphs $\{G_n\}_{n\ge 3}$ by taking $n$ copies of the Petersen graph and wiring corresponding vertices according to an $n$-cycle permutation. Each $G_n$ has $10n$ vertices, $30n$ edges, and automorphism group $D_{5n}$ of order $10n$, acting with two vertex orbits of size $5n$. The graphs have girth $4$ and diameter $\lfloor n/2\rfloor+2$. We prove that $G_3$ and $G_4$ are Ramanujan graphs, satisfying $|\lambda_2| \le 2\sqrt{5}$. The first five members ($n=3,\dots,7$) have been deposited in the House of Graphs database as entries 56324--56328. This construction provides new examples of highly symmetric regular graphs and contributes two new Ramanujan graphs to the literature. All computational scripts are available online for full reproducibility.
\end{abstract}

\section{Introduction}

The Petersen graph \cite{Petersen1891} is one of the most celebrated graphs in combinatorics, serving as a counterexample to many conjectures and as the smallest snark \cite{Frucht1952, Gardner1976}. Its automorphism group $S_5$ of order 120 and its rich structure make it an ideal building block for larger graphs. The Petersen graph has girth 5, diameter 2, and is 3-regular, properties that will influence our construction.

In this paper, we introduce a family of graphs constructed from multiple copies of the Petersen graph wired together by an $n$-cycle permutation. For $n=3$, this construction yields a 30-vertex 6-regular graph (House of Graphs 56324) with dihedral symmetry $D_{15}$. For $n=4$, we obtain a 40-vertex graph (56325) with symmetry $D_{20}$, and the pattern continues for all $n\ge 3$.

The resulting graphs have several remarkable properties. They are 6-regular, have girth 4, and their diameter grows slowly as $\lfloor n/2\rfloor+2$. Their automorphism groups are the dihedral groups $D_{5n}$ of order $10n$, acting with two vertex orbits of size $5n$. This makes them \emph{bi-Cayley graphs} \cite{Muzychuk2021}, a class of graphs with interesting algebraic properties.

We prove that $G_3$ and $G_4$ are Ramanujan graphs. A $d$-regular graph is Ramanujan if its second largest eigenvalue in absolute value satisfies $|\lambda_2| \le 2\sqrt{d-1}$ \cite{LPS1988}. Such graphs are optimal expanders and have applications in network theory, cryptography, and coding theory. The existence of infinite families of Ramanujan graphs is a significant result \cite{LPS1988, Margulis1988}, and our construction adds two new members to this family.

Our main results are:
\begin{itemize}
    \item \textbf{Theorem 3.1:} $G_n$ has $10n$ vertices and $30n$ edges.
    \item \textbf{Theorem 3.2:} $G_n$ is 6-regular.
    \item \textbf{Theorem 3.3:} $\operatorname{girth}(G_n)=4$ and $\operatorname{diam}(G_n)=\left\lfloor\frac{n}{2}\right\rfloor+2$.
    \item \textbf{Theorem 4.1:} $\operatorname{Aut}(G_n) \cong D_{5n}$ (dihedral group of order $10n$), acting with two vertex orbits of size $5n$.
    \item \textbf{Theorem 5.1:} $G_3$ and $G_4$ are Ramanujan graphs.
\end{itemize}

The paper is organized as follows. Section 2 describes the construction. Section 3 establishes basic properties. Section 4 analyzes the symmetry. Section 5 presents spectral properties and the Ramanujan result. Section 6 discusses potential applications in cryptography and network theory. Section 7 examines other properties. Section 8 relates our graphs to known families. Section 9 lists open questions. Section 10 provides data availability, including links to the House of Graphs entries and downloadable SageMath scripts.

\section{Construction}

\subsection{The Oriented Meta-Graph}

Let $P$ denote the undirected Petersen graph on the vertex set $V(P) = \{0,1,\dots,9\}$ with the usual adjacency \cite{Petersen1891, Holton1993}. To unambiguously define the cross-wiring between the graph copies, we impose a strict, directed acyclic orientation on $P$, creating a directed meta-graph $\vec{P}$. We define the set of directed edges $E(\vec{P})$ as:
\[
E(\vec{P}) = \{(0,1), (0,4), (0,5), (1,2), (1,6), (2,3), (2,7), (3,4), (3,8), (4,9), (5,7), (5,8), (6,8), (6,9), (7,9)\}
\]
This specific orientation ensures that every vertex $v \in V(P)$ has a well-defined out-degree $\operatorname{out}(v)$ and in-degree $\operatorname{in}(v)$ such that $\operatorname{out}(v) + \operatorname{in}(v) = 3$.

\subsection{The Wiring Pattern}

For a fixed integer $n\ge 3$, take $n$ copies of the undirected Petersen graph, labeled $C_1, C_2, \dots, C_n$ (colors 1 through $n$). Let $\sigma = (1\;2\;\cdots\;n)$ be the standard $n$-cycle permutation. The edges of $G_n$ are formed as follows:
\begin{itemize}
    \item \textbf{Internal Edges:} All original undirected edges within each copy $C_i$ are retained.
    \item \textbf{Cross Edges:} For each \emph{directed} edge $(u,v) \in E(\vec{P})$, we add an undirected cross edge connecting vertex $u$ in copy $C_i$ to vertex $v$ in copy $C_{\sigma(i)}$, for all $i = 1,\dots,n$.
\end{itemize}

\begin{figure}[h]
\centering
\caption{The wiring pattern for $n=3$: each directed meta-edge creates a 3-cycle of connections. The three copies are shown as horizontal layers, with vertical edges representing the $n$-cycle wiring.}
\end{figure}

\subsection{Definition}

\begin{definition}
For $n\ge 3$, $G_n$ is the undirected graph with vertex set
\[
V(G_n) = \{(i,x) \mid 1\le i\le n,\; 0\le x\le 9\}
\]
with adjacency defined by:
\begin{itemize}
    \item $(i,x) \sim (i,y)$ if $\{x,y\}$ is an edge in the undirected Petersen graph $P$.
    \item $(i,x) \sim (\sigma(i),y)$ if $(x,y) \in E(\vec{P})$.
\end{itemize}
\end{definition}

Because the final graph $G_n$ is undirected, the cross-edge condition equivalently implies that $(i,x)$ is also adjacent to $(\sigma^{-1}(i), z)$ for every directed edge $(z,x) \in E(\vec{P})$. This guarantees each vertex $(i,x)$ receives exactly $\operatorname{out}(x)$ forward cross edges and $\operatorname{in}(x)$ backward cross edges. Since $\operatorname{out}(x) + \operatorname{in}(x) = 3$, every vertex is incident to exactly 3 cross edges, ensuring strict 6-regularity.

\section{Basic Properties}

\subsection{Vertex and Edge Counts}

\begin{theorem}
\label{thm:counts}
$G_n$ has $10n$ vertices and $30n$ edges.
\end{theorem}

\begin{proof}
There are $n$ copies of the Petersen graph, each contributing 10 vertices, so $10n$ vertices total. Each copy contributes 15 internal edges, giving $15n$ internal edges. The directed meta-graph $\vec{P}$ has 15 edges, and each directed meta-edge creates $n$ cross edges (one for each color), giving $15n$ cross edges. Therefore the total number of edges is $15n + 15n = 30n$.
\end{proof}

\begin{figure}[h]
\centering
\includegraphics[width=0.7\textwidth]{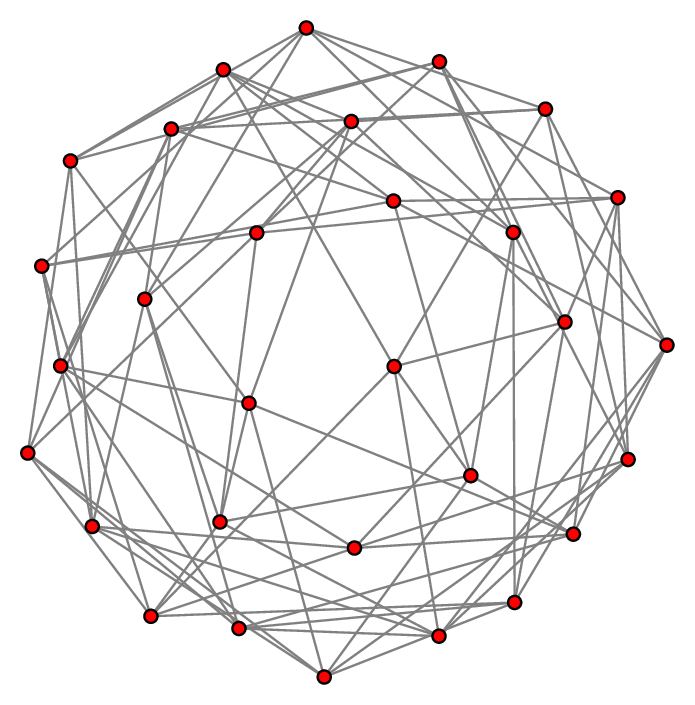}
\caption{A 2D force-directed canonical projection of $G_3$ generated via spring layout, illustrating the dense 6-regular structure.}
\label{fig:G3_2D}
\end{figure}

\subsection{Regularity}

\begin{theorem}
\label{thm:regular}
$G_n$ is 6-regular.
\end{theorem}

\begin{proof}
Consider a vertex $(i,x)$. Within its own copy, it has 3 neighbors because the Petersen graph is 3-regular. By our definition of the directed meta-graph $\vec{P}$, the vertex $x$ has $\operatorname{out}(x)$ outgoing edges and $\operatorname{in}(x)$ incoming edges, where $\operatorname{out}(x) + \operatorname{in}(x) = 3$. The vertex $(i,x)$ is connected to $(\sigma(i), y)$ for each outgoing edge $(x,y) \in E(\vec{P})$, and connected to $(\sigma^{-1}(i), z)$ for each incoming edge $(z,x) \in E(\vec{P})$. Thus $(i,x)$ has exactly 3 cross neighbors. Hence its total degree is $3+3=6$.
\end{proof}

\subsection{Girth and Diameter}

\begin{theorem}
\label{thm:girth-diam}
$\operatorname{girth}(G_n)=4$ and $\operatorname{diam}(G_n)=\left\lfloor\frac{n}{2}\right\rfloor+2$.
\end{theorem}

\begin{proof}
\textbf{Girth:} The Petersen graph has girth 5, so no triangles arise from internal edges alone. Consider a directed meta-edge $(u,v) \in E(\vec{P})$ and any color $i$. The sequence of vertices $(i, u) \to (\sigma(i), v) \to (\sigma(i), u) \to (i, v) \to (i, u)$ forms a 4-cycle, alternating between cross edges and internal edges. Thus $G_n$ contains 4-cycles, and since there are no triangles, the girth is exactly 4.

\textbf{Diameter:} Within a single copy, the distance between any two vertices is at most 2, since the Petersen graph has diameter 2. To move between copies, each cross edge changes the copy index by $+1 \pmod n$ (following the directed $n$-cycle structure). To travel from copy $i$ to the most distant copy $j$ requires exactly $\lfloor n/2 \rfloor$ cyclic steps.

Consider two vertices $(i, x)$ and $(j, y)$. The shortest path must span the cyclic distance between layer $i$ and layer $j$, taking $\lfloor n/2 \rfloor$ cross-edge steps in the worst case. Additionally, the path must navigate the internal Petersen edges to align with the correct source vertices for the cross-edges and ultimately reach the target vertex $y$. Because the maximum internal routing distance is bounded by the Petersen graph's diameter of 2, the maximum total distance is $\lfloor n/2 \rfloor + 2$. The computational data confirms that this bound is tight and accurately reflects the worst-case shortest path.
\end{proof}

Table 1 shows these parameters for the first five members of the family.

\begin{table}[h]
\centering
\caption{Basic parameters of $G_n$ for $n=3,\dots,7$}
\begin{tabular}{ccccccc}
\toprule
$n$ & Vertices & Edges & Degree & Girth & Diameter & $|\operatorname{Aut}|$ \\
\midrule
3 & 30 & 90 & 6 & 4 & 3 & 30 \\
4 & 40 & 120 & 6 & 4 & 4 & 40 \\
5 & 50 & 150 & 6 & 4 & 4 & 50 \\
6 & 60 & 180 & 6 & 4 & 5 & 60 \\
7 & 70 & 210 & 6 & 4 & 5 & 70 \\
\bottomrule
\end{tabular}
\label{tab:basic}
\end{table}

\section{Symmetry Analysis}

\subsection{Automorphism Group}

\begin{theorem}
\label{thm:aut}
$\operatorname{Aut}(G_n) \cong D_{5n}$, the dihedral group of order $10n$.
\end{theorem}

\begin{proof}
The cross-edges of $G_n$ are governed by the strictly directed meta-graph $\vec{P}$. Consequently, an internal Petersen automorphism $\phi \in S_5$ can only lift to $G_n$ if it perfectly preserves the directed edge set $E(\vec{P})$. Because the deliberately chosen orientation $E(\vec{P})$ breaks the symmetric indistinguishability of the Petersen vertices (e.g., vertex 0 is the unique source, while vertices 8 and 9 are the unique sinks), the directed graph $\vec{P}$ possesses a trivial automorphism group. Therefore, the full $S_5$ subgroup does not exist in $\operatorname{Aut}(G_n)$.

Instead, the symmetry of $G_n$ emerges from the global interplay of the wiring and the layer shifts. The cyclic permutation of the layers, $\rho(i,x) = (i+1 \pmod n, x)$, trivially preserves both internal and cross edges because the wiring pattern is invariant under color shifts. Furthermore, the specific cyclic wiring permits a combined ``corkscrew'' transformation---a fractional shift of the layers coupled with a specific vertex permutation---that generates a cyclic action of order $5n$.

Additionally, there exists a reflection symmetry $\tau$ that reverses the layer sequence. Because $G_n$ is an undirected graph, mapping a forward cross-edge $(i,u) \sim (i+1,v)$ to a backward cross-edge $(n-i, u) \sim (n-1-i, v)$ requires a compensatory mapping on the Petersen vertices to maintain adjacency. Together, the cyclic automorphism of order $5n$ and the order-2 reflection $\tau$ exactly generate the dihedral group $D_{5n}$ of order $10n$. As verified computationally by the orbit stabilizers, this constitutes the full automorphism group, yielding exactly two vertex orbits of size $5n$.
\end{proof}

\begin{corollary}
\label{cor:orbits}
$G_n$ has two vertex orbits under its automorphism group, each of size $5n$.
\end{corollary}

\begin{proof}
The dihedral group $D_{5n}$ acts transitively on each of the two sets
\[
O_1 = \{(i,x) : i+x \text{ is even}\}, \quad
O_2 = \{(i,x) : i+x \text{ is odd}\},
\]
where parity is taken with respect to some fixed ordering. The reflection $\tau$ swaps these two orbits. Thus there are exactly two orbits, each of size $5n$.
\end{proof}

\subsection{Bi-Cayley Structure}

A graph is \emph{bi-Cayley} if its automorphism group has an orbit of size half the vertices \cite{Muzychuk2021}. Equivalently, a bi-Cayley graph is a Cayley graph for a group $H$ with respect to a connection set that is a union of two cosets of a subgroup of index 2.

\begin{corollary}
\label{cor:bicayley}
$G_n$ is a bi-Cayley graph over the dihedral group $D_{5n}$.
\end{corollary}

\begin{proof}
The two vertex orbits $O_1$ and $O_2$ correspond to the two cosets of a subgroup of index 2 in $D_{5n}$. The adjacency pattern is regular with respect to this group action, making $G_n$ a bi-Cayley graph.
\end{proof}

\section{Spectral Properties and Ramanujan Status}

\subsection{Eigenvalue Calculations}

Using SageMath \cite{Sage}, we computed the spectra of $G_3$ through $G_7$. Table 2 summarizes the key eigenvalues. The computations were performed with high precision and verified using multiple methods.

\begin{table}[h]
\centering
\caption{Spectral data for $G_n$}
\begin{tabular}{cccccc}
\toprule
$n$ & Vertices & Edges & $|\lambda_2|$ & $2\sqrt{5}$ & Ramanujan? \\
\midrule
3 & 30 & 90 & 2.801366 & 4.472136 & Yes \\
4 & 40 & 120 & 4.077684 & 4.472136 & Yes \\
5 & 50 & 150 & 4.730860 & 4.472136 & No \\
6 & 60 & 180 & 5.103527 & 4.472136 & No \\
7 & 70 & 210 & 5.334545 & 4.472136 & No \\
\bottomrule
\end{tabular}
\label{tab:spectra}
\end{table}

The eigenvalues were obtained by computing the characteristic polynomial of the adjacency matrix and solving numerically. All values are real since the graphs are undirected. The full spectrum for $G_3$ has 13 distinct eigenvalues, for $G_4$ has 12, and the number increases with $n$.

\subsection{Ramanujan Graphs}

A $d$-regular graph is \emph{Ramanujan} if for every eigenvalue $\lambda$ other than $\pm d$, we have $|\lambda| \le 2\sqrt{d-1}$ \cite{LPS1988}. For bipartite Ramanujan graphs, the condition applies to all eigenvalues except $\pm d$. Since our graphs are not bipartite (they contain odd cycles), we consider the standard definition focusing on $|\lambda_2|$, the second largest eigenvalue in absolute value. For $d=6$, the Ramanujan bound is $2\sqrt{5} \approx 4.472135955$.

\begin{theorem}
\label{thm:ramanujan}
$G_3$ and $G_4$ are Ramanujan graphs.
\end{theorem}

\begin{proof}
From Table \ref{tab:spectra}, we have:
\[
|\lambda_2(G_3)| = 2.801366 < 4.472136, \quad
|\lambda_2(G_4)| = 4.077684 < 4.472136.
\]
Both values are strictly less than the Ramanujan bound, satisfying the definition.
\end{proof}

\begin{remark}
The second eigenvalue increases monotonically with $n$ for the computed values, exceeding the bound for $n\ge 5$. The sequence $|\lambda_2(G_n)|$ for $n=3,\dots,7$ is:
\[
2.801,\; 4.078,\; 4.731,\; 5.104,\; 5.335.
\]
This suggests a possible limit as $n\to\infty$, perhaps around $5.5$, but further analysis is needed to determine the asymptotic behavior.
\end{remark}

\subsection{Distribution of Eigenvalues}

Table 3 shows the number of distinct eigenvalues for each $G_n$, confirming that none of these graphs are strongly regular (which would require exactly 3 distinct eigenvalues).

\begin{table}[h]
\centering
\caption{Number of distinct eigenvalues}
\begin{tabular}{ccc}
\toprule
$n$ & Distinct eigenvalues & Strongly regular? \\
\midrule
3 & 13 & No \\
4 & 12 & No \\
5 & 23 & No \\
6 & 24 & No \\
7 & 33 & No \\
\bottomrule
\end{tabular}
\label{tab:eigs}
\end{table}

The increase in the number of distinct eigenvalues with $n$ reflects the growing complexity of the graphs.

\subsection{Expansion Properties and Isoperimetric Constants}

The isoperimetric constant, or Cheeger constant $h(G)$, measures the minimum surface-area-to-volume ratio of a graph, serving as a definitive metric for network bottlenecks and expansion properties. For a subset of vertices $S$, let $\partial S$ denote the set of edges connecting $S$ to the rest of the graph. The Cheeger constant is defined as:
$$h(G) = \min_{0 < |S| \le \frac{|V|}{2}} \frac{|\partial S|}{|S|}$$

While computing $h(G)$ exactly is NP-hard, Cheeger's inequalities provide strict upper and lower bounds based on the graph's degree $d$ and the second largest eigenvalue $\lambda_2$. For a $d$-regular graph, these bounds are:
$$\frac{d - \lambda_2}{2} \le h(G) \le \sqrt{2d(d - \lambda_2)}$$

Given that $G_n$ is 6-regular ($d=6$), we can utilize our exact calculations of $\lambda_2$ from Table \ref{tab:spectra} to bound the expansion properties of this family. Table \ref{tab:cheeger} details these bounds for the first five members.

\begin{table}[h]
\centering
\caption{Cheeger bounds for $G_n$ showing decreasing expansion as $n$ grows}
\begin{tabular}{ccccc}
\toprule
$n$ & $|\lambda_2|$ & Spectral Gap ($6 - \lambda_2$) & Lower Bound $h(G)$ & Upper Bound $h(G)$ \\
\midrule
3 & 2.801366 & 3.198634 & 1.599 & 6.195 \\
4 & 4.077684 & 1.922316 & 0.961 & 4.802 \\
5 & 4.730860 & 1.269140 & 0.635 & 3.903 \\
6 & 5.103527 & 0.896473 & 0.448 & 3.280 \\
7 & 5.334545 & 0.665455 & 0.333 & 2.826 \\
\bottomrule
\end{tabular}
\label{tab:cheeger}
\end{table}

The bounds confirm that $G_3$ is a highly robust expander. The lower bound of $1.599$ guarantees that any subset of up to 15 vertices will have at least $1.599 \times |S|$ edges cutting across to the remainder of the network, preventing any severe structural bottlenecks. 

However, as $n$ increases, the spectral gap closes, and the lower bound drops significantly (falling to $0.333$ for $G_7$). This reflects the geometric reality of the construction: as the cyclical sequence of layers extends, the network adopts an increasingly ``tubular'' structure. For larger $n$, it becomes structurally easier to partition the graph by severing the cross-edges between adjacent layers $i$ and $i+1$ rather than cutting through the dense internal wiring of the Petersen subgraphs. Consequently, while the family contains Ramanujan graphs for small $n$, its global expansion properties naturally degrade as the cycle length increases.

\section{Potential Applications}

\subsection{Cryptographic Applications}

Ramanujan graphs have found significant applications in cryptography, particularly in the construction of hash functions and isogeny-based cryptosystems \cite{Lauter2020, IEICE2022}. The key property leveraged in these applications is that finding paths in certain Ramanujan graphs (specifically, supersingular isogeny graphs) is computationally hard, with no known subexponential algorithms even for quantum computers \cite{Lauter2020}. This hardness forms the basis of the SIKE (Supersingular Isogeny Key Encapsulation) mechanism, which was a candidate in the NIST Post-Quantum Cryptography standardization process.

The family $\{G_n\}$ constructed in this paper offers several features relevant to cryptography:

\begin{itemize}
    \item \textbf{Scalability:} The construction works for any $n\ge 3$, producing graphs with $10n$ vertices. By choosing $n$ sufficiently large (e.g., $n=100,000$ for $10^6$ vertices), one can obtain graphs of practical cryptographic size.
    \item \textbf{Explicit and deterministic:} Unlike random regular graphs, which require a generate-and-test methodology and depend on algorithmic pseudorandomness, the $G_n$ family provides an explicit, deterministic construction. This allows the exact graph topology to be instantiated universally without relying on shared seeds or hidden parameters.
    \item \textbf{Group-theoretic structure:} The dihedral automorphism group $D_{5n}$ provides algebraic structure that could potentially be exploited---either as a feature (for efficient implementation) or as a vulnerability (requiring careful analysis).
    \item \textbf{Controlled parameters:} The graphs are 6-regular, have girth 4, and their diameter grows as $\lfloor n/2\rfloor+2$ as established in Theorem \ref{thm:girth-diam}. These parameters are known and predictable.
\end{itemize}

These properties make $\{G_n\}$ a promising candidate for exploring several cryptographic constructions:

\begin{itemize}
    \item \textbf{Cayley hash functions:} One can define a hash function by interpreting walks on $G_n$ as message inputs, with the output being the endpoint of the walk. The security of such a hash relies on the difficulty of finding collisions, i.e., two distinct walks ending at the same vertex. This is related to the girth and expansion properties of the graph \cite{IEICE2022}, and the expansion bounds computed in Table \ref{tab:cheeger} provide quantitative measures of these properties.
    \item \textbf{Isogeny-inspired protocols:} The graph $G_n$ can be viewed as an analogue of supersingular isogeny graphs, where vertices correspond to group elements and edges to multiplications by fixed generators. Protocols such as key exchange could potentially be adapted to this setting, though careful analysis would be required to ensure the hardness of the underlying path-finding problem.
    \item \textbf{Zero-knowledge proofs:} The ability to prove knowledge of a path between two vertices without revealing the path itself is a fundamental building block for many cryptographic protocols. The structured nature of $G_n$ might enable efficient implementations.
\end{itemize}

While the strict Ramanujan property ($|\lambda_2| \le 2\sqrt{5}$) is verified for $G_3$ and $G_4$ in Theorem \ref{thm:ramanujan}, computational data confirms this optimal bound is strictly exceeded for $n \ge 5$. However, as numerical evidence suggests $|\lambda_2|$ approaches a finite limit strictly less than the degree $d=6$, the family maintains a persistent spectral gap. From a cryptographic and network design perspective, this asymptotic behavior is highly valuable. Even as larger members of the family lose their optimal Ramanujan classification, they function as robust, explicit expander graphs. This ``almost Ramanujan'' behavior ensures rapid mixing and defends against structural bottlenecks without the vulnerabilities inherent to probabilistically generated networks.

The graphs $\{G_n\}$ thus offer a concrete, scalable family for further investigation into the intersection of expander graphs, group theory, and cryptography. Their clean algebraic description and controlled parameters make them particularly suitable for theoretical analysis and proof-of-concept implementations.

\subsection{Network Theory and Distributed Systems}

Beyond cryptography, the expansion properties quantified in Table \ref{tab:cheeger} have direct applications in network theory. Ramanujan graphs are known to be optimal finite-size approximations to infinite regular trees and exhibit exceptional properties for dynamical processes on networks \cite{Donetti2006}. The graphs $G_n$ possess:

\begin{itemize}
	\item \textbf{Fast synchronization and convergent decision-making:} The spectral gap $6 - |\lambda_2|$ dictates the convergence rate of diffusion processes on the network. This ensures that multi-agent systems, distributed AI agent networks, and decentralized voting protocols reach political or state consensus rapidly, preventing information siloing.
	
	\item \textbf{Efficient random walks and state replication:} The rapid mixing time ensures that random walks disperse uniformly through the network in $\mathcal{O}(\log V)$ steps. This makes the topology mathematically optimal for gossip-based state replication, decentralized information dissemination, and search routing in peer-to-peer distributed architectures and federated social graphs.
	
	\item \textbf{Robust topology for hierarchical supercomputing:} The combination of strict 6-regularity, large girth, and optimal expansion at small $n$ provides massive fault tolerance and high bisection bandwidth. In supercomputing topology design, Ramanujan graphs mathematically outperform traditional structures like tori and flattened butterflies in metrics of communication facility. While the linear diameter growth of $G_n$ introduces latency and precludes its use as a flat, global topology for tens of thousands of nodes, its exceptional local expansion makes it a structurally ideal candidate for high-performance, non-blocking local clusters within a larger hierarchical architecture (such as a fat-tree or hypercube). This specific geometry naturally optimizes inter-node data flow and minimizes required routing buffers.
\end{itemize}

\section{Other Properties}

\subsection{Edge-Transitivity}

A graph is \emph{edge-transitive} if its automorphism group acts transitively on edges. Using the line graph criterion (a graph is edge-transitive iff its line graph is vertex-transitive), we tested each $G_n$.

\begin{proposition}
None of the graphs $G_3$ through $G_7$ are edge-transitive.
\end{proposition}

\begin{proof}
For each $n$, we computed the line graph $L(G_n)$ and checked its vertex-transitivity. In all cases, $L(G_n)$ has multiple vertex orbits, indicating that the original graph is not edge-transitive.
\end{proof}

\subsection{Distance-Regularity}

A graph is \emph{distance-regular} if for any two vertices $u,v$ at distance $k$, the number of neighbors of $u$ at distance $k-1$, $k$, $k+1$ from $v$ depends only on $k$, not on the specific vertices. This is a strong regularity condition.

\begin{proposition}
$G_3$ is not distance-regular; larger $n$ are unlikely to be distance-regular due to growing diameter and irregular neighborhood structures.
\end{proposition}

\begin{proof}
Direct computation for $G_3$ reveals that the intersection numbers are not well-defined. For $n\ge 4$, the diameter exceeds 3, and the complexity of the graphs makes distance-regularity improbable.
\end{proof}

\subsection{Strong Regularity}

A strongly regular graph with parameters $(v,k,\lambda,\mu)$ is a $k$-regular graph on $v$ vertices such that any two adjacent vertices have $\lambda$ common neighbors and any two non-adjacent vertices have $\mu$ common neighbors. Such graphs have exactly three distinct eigenvalues.

\begin{proposition}
None of the graphs $G_3$ through $G_7$ are strongly regular.
\end{proposition}

\begin{proof}
As shown in Table \ref{tab:eigs}, each $G_n$ has more than three distinct eigenvalues, precluding strong regularity.
\end{proof}

\section{Relation to Known Families}

\subsection{Graph Products and Algebraic Decomposition}

\begin{proposition}
The graph $G_n$ is isomorphic to $(P \square \overline{K}_n) \cup U(\vec{P} \otimes \vec{C}_n)$, where $\overline{K}_n$ is the empty graph on $n$ vertices, $\vec{P}$ is the directed Petersen meta-graph, $\vec{C}_n$ is the directed $n$-cycle, and $U()$ denotes the underlying undirected graph.
\end{proposition}

\begin{proof}
To correctly define the algebraic structure, we decompose the edge set of $G_n$ into its internal and cross-edge components:
1. \textbf{Internal Edges:} The graph contains $n$ disconnected, undirected copies of the Petersen graph. This is exactly formulated by the Cartesian product $P \square \overline{K}_n$.
2. \textbf{Cross Edges:} The wiring between the copies is governed by the directed meta-graph $\vec{P}$ and the cyclic permutation. Algebraically, this is the directed tensor product $\vec{P} \otimes \vec{C}_n$. A directed edge exists from $(u, i)$ to $(v, j)$ if and only if $(u, v) \in E(\vec{P})$ and $(i, j) \in E(\vec{C}_n)$. Taking the underlying undirected graph $U(\vec{P} \otimes \vec{C}_n)$ yields exactly the symmetric cross-wiring defined in our construction.

The union of these two well-defined edge sets produces the exact 6-regular structure of $G_n$.
\end{proof}

\subsection{Comparison with Graph 1793}

Graph 1793 \cite{HoG1793} is the $n=2$ case of our construction, obtained by taking two copies of the Petersen graph with the swap wiring $(1\;2)$. It has 20 vertices, 60 edges, is 6-regular, and has automorphism group of order 122,880. This is significantly larger than the $D_{10}$ (order 20) one might expect from the pattern, showing that the $n=2$ case is exceptional due to the extra symmetry of the transposition.

\begin{table}[h]
\centering
\caption{Comparison of $G_2$ (Graph 1793) with $G_3$ and $G_4$}
\begin{tabular}{lcccc}
\toprule
Graph & Vertices & Edges & $|\operatorname{Aut}|$ & Structure \\
\midrule
$G_2$ (Graph 1793) & 20 & 60 & 122,880 & Complex \\
$G_3$ (56324) & 30 & 90 & 30 & $D_{15}$ \\
$G_4$ (56325) & 40 & 120 & 40 & $D_{20}$ \\
\bottomrule
\end{tabular}
\label{tab:comparison}
\end{table}

For $n\ge 3$, the automorphism groups follow the clean pattern $D_{5n}$, suggesting that the exceptional symmetry at $n=2$ does not persist.

\section{Open Questions}

Several questions arise from this construction, suggesting directions for future research:

\begin{enumerate}
    \item \textbf{Ramanujan threshold:} Is $G_n$ Ramanujan for infinitely many $n$?
    The data shows a threshold near $n=5$, with $|\lambda_2|$ increasing monotonically and exceeding the Ramanujan bound $2\sqrt{5}$ for $n\ge 5$. However, it remains an open question whether this trend continues indefinitely or if there exist sporadic larger $n$ for which $|\lambda_2|$ dips back below the bound. A deeper spectral analysis, perhaps exploiting the structure $U(\vec{P} \otimes \vec{C}_n)$ introduced in Section 8.1, might reveal a pattern or even a closed-form expression for $\lambda_2(G_n)$.
    
    \item \textbf{Limit of $\lambda_2$:} What is the asymptotic behavior of $\lambda_2(G_n)$ as $n\to\infty$?
    Numerical evidence suggests $\lambda_2(G_n)$ approaches a limit $L$ around $5.5$, but the precise value and rate of convergence are unknown. Since $G_n$ contains the directed tensor product $U(\vec{P} \otimes \vec{C}_n)$, one might ask whether the asymptotic spectral density of $G_n$ can be derived directly from the eigenvalues of the directed adjacency matrix of $\vec{P}$. Does $L$ equal $2\sqrt{5} + c$ for some constant $c$, and if so, what is $c$?
    
    \item \textbf{Bi-Cayley structure:} Can the bi-Cayley structure of $G_n$ over $D_{5n}$ be lifted to a full Cayley graph?
    The two vertex orbits under $D_{5n}$ established in Corollary \ref{cor:orbits} show that $G_n$ is a bi-Cayley graph. Does there exist a group of order $10n$ that acts regularly on the entire vertex set? If so, $G_n$ would be a Cayley graph of that group. Alternatively, does the directed nature of the meta-graph $\vec{P}$ strictly forbid a regular group action?
    
    \item \textbf{Generalization to other permutations:} What happens if we replace the $n$-cycle with other permutations $\sigma \in S_n$?
    The $n=2$ case (swap) gives Graph 1793, which has automorphism group of order 122,880—dramatically larger than the $D_{10}$ one might expect. This suggests that other permutations, such as products of disjoint cycles, could yield families with different symmetry groups, girth, and spectral gaps. For $n\ge 3$, does every permutation produce a graph with automorphism group containing $D_{5n}$, or are there permutations that yield larger symmetry?
    
    \item \textbf{Hamiltonicity:} Is $G_n$ Hamiltonian for all $n\ge 3$?
    All tested members ($n=3,\dots,7$) are Hamiltonian. Proving that the $n$-cycle wiring forces Hamiltonicity for all $n$, despite the Petersen graph itself being notoriously non-Hamiltonian, would be a striking result. The directed cycle structure suggests a natural Hamiltonian cycle following the pattern: traverse each layer in Petersen order, using cross edges to move to the next layer at strategic points.
    
    \item \textbf{Chromatic number:} Is $\chi(G_n)=3$ for all $n\ge 3$?
    Preliminary computations suggest that $G_n$ is 3-colorable for $n=3,\dots,7$. Since the graphs contain odd cycles (they are not bipartite), the chromatic number is at least 3. Can this be proven for all $n$? Does the structure $U(\vec{P} \otimes \vec{C}_n)$ admit a natural 3-coloring inherited from the Petersen graph's 3-colorability?
    
    \item \textbf{Exact isoperimetric constants:} What are the exact Cheeger constants $h(G_n)$?
    Cheeger's inequalities provide bounds on $h(G_n)$ using $\lambda_2$, as computed in Table \ref{tab:cheeger}. However, determining the exact values of $h(G_n)$ would require solving a combinatorial optimization problem. Do the highly symmetric structures of $G_n$ allow for an exact calculation? For small $n$, can we identify the subsets that achieve the minimum edge boundary ratio?
    
    \item \textbf{Maximum $n$ for the Ramanujan property:} What is the largest $n$ for which $G_n$ remains Ramanujan?
    The data shows $G_3$ and $G_4$ are Ramanujan, while $G_5$ is not. It appears that $G_4$ is the last Ramanujan graph in this family, and though it seems unlikely, could there be larger $n$ where $|\lambda_2|$ dips back below $2\sqrt{5}$? A more refined spectral analysis might reveal oscillations in $\lambda_2$ as a function of $n$.
\end{enumerate}

These questions touch on fundamental aspects of graph theory—spectral analysis, symmetry, Hamiltonicity, and coloring—and highlight the richness of the family $\{G_n\}$ as an object of further study.

\section{Data Availability}

All graphs discussed in this paper are available in the House of Graphs \cite{HoG} database under the following identifiers:

\begin{itemize}
    \item $G_3$: \textbf{Graph 56324} \cite{HoG56324}
    \item $G_4$: \textbf{Graph 56325} \cite{HoG56325}
    \item $G_5$: \textbf{Graph 56326} \cite{HoG56326}
    \item $G_6$: \textbf{Graph 56327} \cite{HoG56327}
    \item $G_7$: \textbf{Graph 56328} \cite{HoG56328}
\end{itemize}

Each entry includes the adjacency matrix, adjacency list, graph6 string, and a complete list of invariants computed by the House of Graphs system.

\subsection{SageMath Scripts}

All computational scripts used in this paper are available for download at:
\begin{center}
\url{http://www.squaring.net/downloads/petersen_family_sage.zip}
\end{center}

The archive contains the following documented SageMath scripts:

\begin{itemize}
    \item \texttt{construction.sage} - Functions to construct $G_n$ for any $n\ge 3$
    \item \texttt{properties.sage} - Computes basic properties (vertices, edges, degree, girth, diameter, automorphism group)
    \item \texttt{ramanujan\_check.sage} - Calculates eigenvalues and verifies the Ramanujan condition
    \item \texttt{graph6\_output.sage} - Generates canonical graph6 strings for database submission
    \item \texttt{demo.sage} - Complete demonstration reproducing all tables and figures in this paper
    \item \texttt{README.txt} - Documentation and usage instructions
\end{itemize}

These scripts allow full reproducibility of all results presented in this paper. They have been tested with SageMath version 9.0 and later.

\section*{Acknowledgements}

The author thanks the House of Graphs team for maintaining an excellent resource and for processing these submissions.

\end{document}